\documentclass[12pt,a4paper]{article}
\usepackage{amssymb}
\usepackage{amsmath}
\usepackage{amsthm}
\usepackage{indentfirst}
\usepackage{ot1patch}
\usepackage{psfrag}
\usepackage{rotate}
\usepackage{hyperref}
\usepackage{pdflscape}
\usepackage{multirow}
\usepackage{tikz}

\newtheorem{tw}{Theorem}[section]
\newtheorem{pr}[tw]{Proposition}
\newtheorem{lm}[tw]{Lemma}
\newtheorem{cor}[tw]{Corollary}

\theoremstyle{definition}

\newtheorem{ex}[tw]{Example}

\DeclareMathOperator{\Irr}{Irr}
\DeclareMathOperator{\Sqf}{Sqf}

\DeclareMathOperator{\Prime}{Prime}
\DeclareMathOperator{\Gpr}{Gpr}

\DeclareMathOperator{\lcm}{lcm}

\newcommand{\rpr}{\mbox{\small\rm rpr}}

\author{\L ukasz Matysiak\\
Kazimierz Wielki University\\
Bydgoszcz, Poland \\
lukmat@ukw.edu.pl}
\title{A polynomial composites}

\begin{document}

\maketitle

\begin{abstract}
Polynomial composites were introduced by Anderson, Anderson, and Zafrullah. In this paper we study many different algebraic properties of polynomial composites like ACCP, atomic, SR property. We study relationships between Noetherian polynomial composites certain field extensions.  
\end{abstract}

\begin{table}[b]\footnotesize\hrule\vspace{1mm}
\end{table}

\section{Introduction}
\label{R1}

By the ring we mean a commutative ring with unity. Let $R$ be an itegral domain.
We denote by $R^{\ast}$ the group of all invertible elements of $R$. 

\medskip

The main motivation of this paper is description polynomial composites as algebraic object. 
The related works were started in paper \cite{Matysiak1}, where basic algebraic properties have been investigated. Continued in \cite{Matysiak2}, where the focus was on ACCP properties and atomicity. 
Next, in \cite{Matysiak3}, where we have research related to various classes of rings as polynomial composites. The article \cite{Matysiak4} shows very strong relationships between polynomial composites and field extensions, and therefore relationships with the Galois theory.

\medskip

D.D.~Anderson, D.F.~Anderson, M. Zafrullah in \cite{1} called object $A+XB[X]$ as a composite, where $A\subset B$ are fields. 

\medskip

In 1972 \cite{Gil} and in 1976 \cite{y1} authors considered the structures in the form $D+M$, where $D$ is a domain and $M$ is a maximal ideal of ring $R$ with $D\subset R$. 
Next, Costa, Mott and Zafrullah (\cite{y2}, 1978) considered composites in the form $D+XD_S[X]$, where $D$ be a domain and $D_S$ be a localization of $D$ relative to the multiplicative subset $S$. In 1988 \cite{y5} Anderson and Ryckaert studied classes groups $D+M$.
Zafrullah in \cite{y3} continued research on structure $D+XD_S[X]$ but he showed that if $D$ be a GCD-domain, then the behaviour of $D^{(S)}=\{a_0+\sum a_iX^i\mid a_0\in D, a_i\in D_S\}=D+XD_S[X]$ depends upon the relationship between $S$ and the prime ideals $P$ of $D$ such that $D_P$ be a valuation domain (Theorem 1, \cite{y3}).
Fontana and Kabbaj in 1990 (\cite{y4}) studied the Krull and valuative dimensions of composite $D+XD_S[X]$. 
In 1991 there was an article (\cite{1}) that collected all previous results about composites and authors began to create a further theory about composites creating results.
Of course, let $K\subseteq L$, where $K$, $L$ are fields, polynomial composites can be applied to algebraic geometry by the fact that $K+XL[X]/XL[X] = K$, i.e $XL[X]$ is the maximal ideal of $K+XL[X]$. Also in the basics of economics (\cite{eco}, \cite{eco2}) we can find the use of polynomials, in particular polynomial composites.
In 2009 autors considered a construction of $A+XB[X]$, where $A\subset B$ be an extension of commutative rings, as particular case of the general construction of amalgamated algebras introduced in \cite{amal}, Example 2.5.

In this paper, the considered structures were officially called composites.

\medskip

In the section \ref{R2} we present basic properties of polynomial composites with possible different configurations of rings, domains, fields, etc. We will show some simple properties from commutative algebra for these structures. In the Theorem \ref{t2.2} we will show that every nonzero prime ideal in $A+XB[X]$ ($A$, $B$ are fields with $A\subset B$), every prime ideal different from $XB[X]$ is principal and $A+XB[X]$ is atomic ring. In Corollary \ref{c2.3} and Theorem \ref{t2.4} we have a characterization of irreducible polynomials in composites. By motivation from \cite{Matysiak5} and \cite{Matysiak6} in Theorem \ref{t2.5} and Corollary \ref{c2.6} we show a characterization of square-free polynomials in composites. Moreover, in \cite{Matysiak7} can find a characterization of radical polynomials in composites.

\medskip

In the section \ref{R3} we present many conditions in polynomial composites. Recall, we say that $R$ satysfies ACCP condition (has ACCP) if every ascending chain of principal ideals of $R$ stabilizes. 
Following P.M. Cohn \cite{Cohn} we say that $R$ is atomic if every nonzero nonunit of $R$ can be factored into irreducibles (atoms).
If an integral domain satisfies ACCP, then it is atomic. However, there are atomic domains that do not satisfy ACCP (the first example was constructed by A. Grams in \cite{Grams}).
In a Proposition \ref{p3.1} we will show when a ring $R$ such that $A[X]\subseteq R\subseteq B[X]$ ($A$, $B$ are fields) is ACCP-ring. A complement of this Proposition we have in the Theorem \ref{t3.2}. In the Theorem \ref{t3.4} we have a condition, when $A+X[B]$ has ACCP ($A$ is domain, $B$ is quotient field of $A$).
The domain $R$ is a bounded factorization domain (BFD) if it is atomic and for every nonzero nonunit $x\in R$ there is a positive integer $N$ such that $x=a_1\dots a_n$ for irreducibles $a_1, \dots, a_n\in R$ implies that $n\leqslant N$ (Propositions \ref{p3.6}, \ref{p3.7}).  
We say that $R$ is a half-factorial domain (HFD) if is atomic and any two factorizations of the same nonzero nonunit of $R$ have the same number of irreducibles (counting repetitions) (Propositions \ref{p3.8}, \ref{p3.9}). Clearly a UFD is a HFD, a HFD is also a BFD, and hence satisfies ACCP (see \cite{AndGotti}). 
The domain $R$ is an idf-domain (for irreducible-divisor-finite) if each nonzero element of $R$ has at most a finite number of nonassociate irreducible divisors (Propositions \ref{p3.10}, \ref{p3.11}). 
A domain is called finite factorization domain (FFD) if each nonzero nonunit element has only a finite number of nonassociate divisors (Proposition \ref{p3.12}). A UFD is a FFD, a FFD is a HFD, a FFD is a BFD (see \cite{AndGotti}). In general,

$$\begin{array}{ccccccccc}
	&&HFD\\
	&\mbox{\rotatebox{-135}{$\Leftarrow$}} & \Uparrow & \mbox{\rotatebox{135}{$\Leftarrow$}} \\
	UFD&\Rightarrow &FFD&\Rightarrow &BFD&\Rightarrow ACCP&\Rightarrow& atomic  \\
	&\mbox{\rotatebox{135}{$\Leftarrow$}}&\Downarrow \\
	&&idf
\end{array}$$

Recall that $R$ is an S-domain (S is for Seidenberg) if for each height-one prime ideal $P$ of $R$, $ht P[X]=1$ in $R[X]$ (Proposition \ref{p3.15}).
A commutative ring $R$ is called a Hilbert ring if every prime ideal of $R$ is an intersection of maximal ideals of $R$ (Theorem \ref{t3.16}).

\medskip

The main motivation of the Section \ref{R4} is to answer the following question.

\medskip

\noindent
{\bf Question:}\\
Is there a relationship between certain field extensions $K\subset L$ and polynomial composites $K+XL[X]$?

\medskip

We present a full possible characterization of polynomial composites in the form $K+XL[X]$, where $K$, $L$ are fields, with respect to a given extension with appropriate additional assumptions. We also present a full possible characterization of some extensions of fields $K\subset L$ expressed in the language of polynomial composites $K+XL[X]$ as Noetherian rings with appropriate assumptions.
From Theorem \ref{t4.5} we will know that $K+XL[X]$ is Noetherian iff $K\subset L$ is finite. 
In Propositions \ref{p4.6} -- \ref{p4.11} we can find similar statements with such extensions as algebraic, separable, normal, Galois.
We will remind results of Magid from \cite{Magid} (Theorems \ref{t4.12}, \ref{t4.13}) and we will supplement these statements with additional conditions (Theorems \ref{t4.14} and \ref{t4.15}). 

\medskip

In the last Section we describe an SR-property in polynomial composites. In \cite{Matysiak7} we have general results for any integral domains and monoids. Recall from \cite{Matysiak5} that an $SR$-domain (monoid) is a domain (monoid) in which every square-free element is radical. Naturally, every radical element is square-free, so in $SR$-domain (monoid) the set of square-free is the same like the set of radical elements.

\section{Basic properties}
\label{R2}

The aim of this chapter will be to examine the simplest structural properties of the considered structures.

Consider $A$ and $B$ as rings such that $A\subset B$.
Put $T=A+XB[X]$. The structure defined in this way is called a composite. (The definition comes from \cite{1}). Later results will require more assumptions.

\medskip

Let's look at invertible polynomials.

\begin{pr}
	\label{p2.1}
	Let $f=a_0+a_1X+\dots + a_nX^n\in T$ for any $n\geq 0$. Then $f\in T^{\ast}$ if and only if $a_0\in A^{\ast}$ and $a_1, a_2, \dots , a_n$ are nilpotents.
\end{pr}

\begin{proof}
	We know that if $R$ is a ring then $f=a_0+a_1X+\dots + a_nX^n\in R[X]^{\ast}$ if and only if $a_0\in R^{\ast}$ and $a_1, a_2, \dots , a_n$ are nilpotents. In our Proposition we have $a_1, a_2, \dots , a_n$ are nilpotents. Of course we get $a_0\in A^{\ast}$.
\end{proof}

\begin{tw}
	\label{t2.2}
	Let $A$, $B$ be fields such that $A\subset B$. Let $X$ be an indeterminate over $B$ and let $T=A+XB[X]$. Then	
	\begin{itemize}
		\item[(1) ] every nonzero prime ideal of $T$ is maximal;
		\item[(2) ] every prime ideal $P$ different from $XB[X]$ is principal;
		\item[(3) ] $T$ is atomic.
	\end{itemize} 
\end{tw}

\begin{proof}
	(1) \ \ 
	First note that $XB[X]$ is maximal since $T/XB[X]\cong K$. Let $P$ be a nonzero prime ideal of $T$. Now $X\in P$ implies $(XB[X])^2\subseteq P$ and hence $XB[X]\subseteq P$ so $P=XB[X]$. So suppose that $X\notin P$. Then for $N=\{1, X, X^2, \dots \}$, $P_N$ is a prime ideal in the PID $B[X,X^{-1}]=T_N$. (In fact, $T_P\supseteq B[X,X^{-1}]$ is a DVR (discrete valuation ring).) So $P$ is minimal and is also maximal unless $P\subsetneq XB[X]$. But let $k_nX^n+\dots + k_sX^s\in P$ with $k_n\neq 0$, where $k_n, \dots , k_s\in L$ for every $n, s$. Then $X^{n+1}+k_n ^{-1}k_{n+1}X^{n+2}+\dots +k_n ^{-1}k_sX^s\in P$, so $X\notin P$ implies that $1+k_n ^{-1}k_{n+1}X+\dots +k_n ^{-1}k_sX^{s-n}\in P$, a contradiction. So every nonzero prime ideal is maximal.
	
	\medskip
	
	\noindent
	(2) \ \ 
	If $P$ is different from $XB[X]$, then it contains an element of the form $1+Xf(X)$, where $f(X)\in B[X]$. Now if $1+Xf(X)$ can be factored in $B[X]$ it can be written as $(1+Xg(X))(1+Xh(X))$. Hence $1+Xf(X)$ is irreducible in $T$ if and only if it is irreducible in $B[X]$. 
	
	\medskip
	
	\noindent
	Now let $1+Xf(X)$ be irreducible in $T$ and suppose that $1+Xf(X)\mid h(X)k(X)$ in $T$. Then $1+Xf(X)\mid h(X)k(X)$ in $B[X]$, and so in $B[X]$ we have, say $1+Xf(X)\mid h(X)$. Then, in $B[X]$, $h(X)=(1+Xf(X))d(X)$. Now $d(X)$ can be written as $d(X)=aX^r(1+Xp(X))$. If $r>0, d(X)\in T$, while if $r=0, h(X)=(1+Xf(X))(a(1+Xp(X))$ and $a\in A$ because $h(X)\in T$. In either case, $d(X)\in T$ and so $1+Xf(X)\mid h(X)$ in $T$. Consequently, in $T$ every irreducible element of the type $1+Xf(X)$ is prime.
	
	\medskip
	
	\noindent
	Now since every element of the form $1+Xf(X)$ is a product of irreducible elements of the same form and hence is a product of prime elements, it follows that every prime ideal of $P$ different from $XB[X]$ contains a principal prime and hence is actually principal.
	
	\medskip
	
	\noindent
	(3) \ \ Thus a general element of $T=A+XB[X]$ can be written as $aX^r(1+Xf(X))$, where $a\in B$ (with $a\in A$ if $r=0$) and $1+Xf(X)$ is a product of primes.
\end{proof}

\begin{cor}
	\label{c2.3}
	Let $A$, $B$ be fields such that $A\subset B$. Consider $D=A+XB[X]$. Then 
	$\Irr D=\{aX, a\in B\}\cup
	\{a(1+Xf(X)),a\in A, f\in B[X], 1+Xf(X)\in\Irr B[X]\}.$
\end{cor}

\begin{tw}
	\label{t2.4}
	Consider $T=A+XB[X]$, where $A$, $B$ are fields, $A\subset B$. Then
	$f\in\Irr T$ if and only if $f\in\Irr B[X], f(0)\in A$.
\end{tw}

\begin{proof}
	Suppose that $f\notin\Irr B[X]$ or $f(0)\notin A$.
	If $f(0)\notin A$, then $f\notin T$, so $f\notin\Irr B[X]$.
	Now, assume that $f\notin\Irr B[X]$.
	Then $f=gh$, where $g, h\in B[x]\setminus B$.
	Let $g=a_0+a_1X+\dots +a^nX^n, h=b_0+b_1X+\dots + b_mX^m$.
	We have $f=(a_0+a_1X+\dots +a^nX^n)(b_0+b_1X+\dots + b_mX^m)$.
	Then
	$f=\big(1+\dfrac{a_1}{a_0}X+\dots + \dfrac{a_n}{a_0}X^n\big)
	(a_0b_0+a_0b_1X+\dots +a_0b_mX^m)$,
	where $a_0b_0=f(0)\in A$.
	Now, suppose that $f\notin\Irr T$.
	If $f\notin T$, then $f(0)\notin A$.
	Now, assume that $f\in T$.
	Then we have $f=gh$, where $g, h\in T\setminus A$.
	This implies $g, h\in B[x]\setminus B$.
\end{proof}

\begin{tw}
	\label{t2.5}
	Let $K$ and $L$ be fields such that $K\subset L$ and let $T=K+XL[X]$. Then $\Sqf (T)=(\Sqf (L[X]) \cap T)\cup \{X^2h; h\in\Sqf (L[X]), h(0)\notin\{a^2b; a\in L, b\in K\}\}$.
\end{tw}

\begin{proof}
	Let $f\in\Sqf (T)\setminus\Sqf (L[X])$. There are some $g\in (L[X])\setminus (L[X]^{\ast})$ and $k\in L[X]$ such that $f=g^2k$. Set $c=g(0)$. If $c\neq 0$ then since $c^{-1}g$, $c^2k\in T$ and $f=(c^{-1}g)^2c^2k$, we have $c^{-1}g\in T^{\ast}$ and $g\in L[X]^{\ast}$, a contradiction. Therefore, $c=0$, and thus $f=X^2h$ for some $h\in L[X]$. Since $f\in\Sqf T$, we infer $h(0)=\neq 0$. If $h(0)=a^2b$ for some $a\in L$ and $b\in L$, then $aX$, $a^{-2}h\in T$, and $f=(aX)^2a^{-2}h$, which contradicts the fact that $f\in\Sqf T$. This implies that $h(0)\notin\{a^2b; a\in L, b\in K\}$. Let $r$, $s\in L[X]$ be such that $h=r^2s$. Since $h(0)\neq 0$, we infer $r(0)\neq 0$. Set $d=r(0)$. Then $f=(d^{-1}r)^2d^2sX^2$ and $d^{-1}r$, $d^2sX^2\in T$. Consequently, $d^{-1}r\in T^{\ast}$, and hence $r\in L[X]^{\ast}$. This shows that $h\in\Sqf L[X]$.
	
	\medskip
	
	Since $(L[X]^{\ast})\cap T=T^{\ast}$, it follows that $\Sqf (L[X]) \cap T\subset \Sqf T$. Now let $h\in\Sqf L[X]$ be such that $h(0)\notin\{a^2b; a\in L, b\in K\}$. It remains to show that $X^2h\in\Sqf T$. Clearly, $X^2h\in T$. Let $r$, $s\in T$ be such that $X^2h=r^2s$. Assume that $r\in XL[X]$. Then $r=Xt$ for some $t\in L[X]$, so $h=t^2s$. Therefore, $h(0)=t(0)^2s(0)\in\{a^2b; a\in L, b\in K\}$, a contradiction. Consequently, $r\notin XL[X]$, and thus $s=X^2w$ for some $w\in L[X]$. We infer $h=r^2w$, and hence $r\in (L[X]^{\ast})\cap T=T^{\ast}$.
\end{proof}

\begin{cor}
	\label{c2.6}
	Let $K$ and $L$ be fields such that $K\subset L$ and let $T=K+XL[X]$. Then $\Sqf T=\Sqf(L[X])\cap T$ iff $L=\{a^2b; a\in L, b\in K\}$. In particular, if $L$ is algebraically closed, then $\Sqf T=\Sqf(L[X])\cap T.$
\end{cor}

\begin{proof}
	It follows from Proposition \ref{t2.3} that if $L=\{a^2b; a\in L, b\in K\}$, then $\Sqf T=\Sqf(L[X])\cap T$. Now let $L\neq\{a^2b; a\in L, b\in K\}$. There is some $c\in L\setminus \{a^2b; a\in L, b\in K\}$. By Proposition \ref{t2.3}, we have $X^2x\notin\Sqf(L[X])$, and thus $\Sqf T\neq\Sqf(L[X])\cap T$. Finally, if $L$ is algebraically closed, then $L=\{a^2; a\in L\}$ and the statement follows.
\end{proof}

\begin{ex}
	\label{e2.7}
	Let $L$ be a field with char $(L)=2$ such that $L$ is not perfect, let $K$ be the prime subfield of $L$ and $T=K+XL[X]$. Then $\Sqf T\neq \Sqf(L[X])\cap T$.
\end{ex}

\begin{proof}
	Since char $(L)=2$ and $L$ is not perfect, we have $L\neq\{a^2; a\in L\}$. Since $K=\{0,1\}$, this implies that $L\neq\{a^2b; a\in L, b\in K\}$. It is an immediate consequence of Corollary \ref{c2.4} that $\Sqf T\neq \Sqf(L[X])\cap T$.
\end{proof}

In particular, if $T=\mathbb{R}+X\mathbb{C}[X]$, then $\Irr T=\{a+bX; a\in\mathbb{R}, b\in\mathbb{C}\setminus\{0\}\}$ and $\Sqf T=\{a\prod_{b\in B} (1+bX); a\in\mathbb{R}\setminus\{0\}, B\subset\mathbb{C}, B \text{ is finite}\}\cup \{aX\prod_{b\in B} (1+bX); a\in\mathbb{C}\setminus\{0\}, B\subset\mathbb{C}, B \text{ is finite}\}$.

If $L$ and $K$ are finite fields and it is a proper extension, then $K+XL[X]$ is a non-factorial ACCP domain (see \cite{0}, \cite{x9}).

Recall that a ring $R$ satisfies ACCP if each chain of principal ideals of $R$ is stabilize.

\begin{ex}
	\label{e2.8}
	In \cite{1} Example 5.1 showed an example of an integral domain $R$ which satisfies ACCP, but whose integral closure does not satisfy ACCP. It mean $R=\mathbb{Z}+X\overline{\mathbb{Z}}[X]$, where $\overline{\mathbb{Z}}$ be the ring of all algebraic integers. $R$ satisfies ACCP. For if not, then there is an infinite properly ascending chain of pricipal ideals of $R$. Since the degrees of the polynomials generating these principal ideals are nonincreasing, the degrees eventually stabilize. The principal ideals in $\overline{\mathbb{Z}}$ generated by the leading coefficients of these polynomials gives an infinite ascending chain $a_1\overline{\mathbb{Z}}\subsetneq a_2\overline{\mathbb{Z}}\subsetneq ...$ where each $a_n/a_{n+1}\in\mathbb{Z}$. Thus all $a_n\in\mathbb{Q}[a_1]$. Let $A=\overline{\mathbb{Z}}\cap\mathbb{Q}[a_1]$. Then $a_1A\subsetneq a_2A\subsetneq\dots \subsetneq A$, a contradiction since $A$ is Dedekind.
\end{ex}

\begin{ex}
	\label{e2.9}
	Let $R=\mathbb{R}+X\mathbb{C}[X]$. So $R$ is a HFD, so has ACCP, then atomic.
\end{ex}

\section{ACCP, BFD, HFD, idf, FFD, S, Hilbert properties}
\label{R3}

Note that for $R$ a ring between $A[X]$ and $B[X]$ ($A$, $B$ are fields such that $A\subset B$), $R$ has ACCP if and only if for every $n\geqslant 0$, any ascending chain of principal ideals generated by polynomials of degree $n$ terminates. If $B$ be a quotient field of $A$, the Proposition 5.2 \cite{1} may be used to show that ring $R$ satisfies ACCP. But we can minimalize assumptions by composites.

\begin{pr}
	\label{p3.1}
	Let $A$, $B$ are fields such that $A\subset B$. Let $R$ be a ring with $A[X]\subseteq R\subseteq B[X]$. Then $R$ has ACCP if and only if $R\cap B$ has ACCP and for each ascending chain of polynomials $f_1R\subseteq f_2R\subseteq f_3R\subseteq\dots $ where $f_i\in R$ all have the same degree, then there is $d\in (R\cap B)\setminus\{0\}$ such that $df_i\in (R\cap B)[X]$. 
\end{pr}

\begin{proof}
	($\Rightarrow$) 
	Since $(R\cap B)^{\ast}=R^{\ast}\cap B$, $R$ has ACCP implies $R\cap B$ has ACCP. The chain $f_1R\subseteq f_2R\subseteq\dots$ be stationary, say $f_nR=f_{n+1}R=\dots$. So $f_{n+1}=u_if_i$, where $u_i$ is a unit of $R\cap B$. Since $f_n\in B[X]$, there exists a $0\neq a\in A\subseteq R\cap B$ with $af_n\in A[X]\subseteq R$. But then coefficients of $af_{n+1}=u_idf_n$ all lie in $R\cap B$.
	
	\medskip
	
	\noindent
	($\Leftarrow$)
	Let $f_1R\subseteq f_2R\subseteq\dots$ be an ascending chain in $R$. Since $\deg f_{i+1}\leqslant \deg f_i$, eventually $f_i$ have the same degree, so without loss of generality, we can assume that $\deg f_1=\deg f_2=\dots $. By hypothesis there exists $0\neq a\in R\cap B$ with $af_i\in(R\cap B)[X]$. Now $f_iR\subseteq f_{i+1}R$ implies $f_i=f_{i+1}b$, where $b\in R$ has degree $0$, so $b\in R\cap B$. Hence $af_i(R\cap B)[X]\subseteq af_{i+1}(R\cap B)[X]$. But $R\cap B$ has ACCP and hence so does $(R\cap B)[X]$. So for large $n$, $f_n(R\cap B)[X]=f_{n+1}(R\cap B)[X]=\dots$, and hence $f_nR=f_{n+1}R=\dots$. 
\end{proof}

A Theorem \ref{t3.2} shows that between $A[X]$ and $A+XB[X]$ we can find a structure which satysfying ACCP condition. 

\begin{tw}
	\label{t3.2}
	Let $A$, $B$ are fields such that $A\subset B$. Let $C$ be a domain such that $A[X]\subseteq C\subseteq A+XB[X]$. Suppose that for each $n\geq 0$, there exists $a_n\in A\setminus\{0\}$ so that $a_nf\in A[X]$ for all $f\in C$ with $\deg f\leq n$. Then $C$ has ACCP if and only if $A$ has ACCP.
\end{tw}

\begin{proof}
	This is a special case of Proposition \ref{p2.1}. The second part of this proof: let us recall $R$ a bounded factorization domain (BFD) if for each nonzero nonunit $a\in R$, there exists a positive integer $N(a)$, so that if $a=a_1\dots a_s$ where each $a_i$ is nonunit, then $s\leqslant N(a)$. It is very known fact that a BFD has ACCP but the converse is false. In the proof suppose that $A$ is a BFD. Let $0\neq f\in C$ have a degree $n$ and leading coefficient $b$. Write $f=g_1\dots g_sg_{s+1}\dots g_m$, where $g_1, \dots , g_m\in C$ are nonunits with $g_1, \dots , g_s\in A$ and $g_{s+1}, \dots , g_m\in C$ have a degree $\geqslant 1$. Now $g_{s+1}\dots g_m$ has a degree $n$, so $m-s\leqslant n$. Also, $r_ng_{s+1}\dots g_m\in A[X]$, say it has leading coefficient $c\in A$. Then $r_nb=g_q\dots g_sc$. But $A$ is a BFD, so there is a bound on the number of factors for $r_nb$ and hence on $s$. Thus the $m$ if $f=g_1\dots g_sg_{s+1}\dots g_m$ has an upper bound. Conversely, if $C$ be a BFD, easy to see that $A$ be a BFD without any additional hypothesis on $C$. 
\end{proof}

\medskip

\begin{pr}
	\label{p3.3}
	Let $A$ be a domain. Then $A$ has ACCP if and only if $A[X]$ has ACCP.
\end{pr}

\begin{proof}
	The equivalene can be easily proved by comparing degrees (resp., orders).
\end{proof}

It also turn out that ACCP property moves between $A$ and $A+XB[X]$. This is important because we do not have to choose a general polynomial, and we can limit the inclusion to the smallest composite needed. Such a significant limitation of a polynomial to a composite is important, e.g. in Galois theory in commutative rings.

\begin{tw}
	\label{t3.4}
	Let $A$ be an integral domain with quotient field $B$. An $A$ has ACCP if and only if $A+XB[X]$ has ACCP.
\end{tw}

\begin{proof}
	From Theorem \ref{t3.2} we have $A[X]\subseteq A+XB[X]\subseteq A+XK[X]$, where $K$ be an overfield of $B$. We can to prove that for each $n\geq 0$, there exists $a_n\in A\setminus\{0\}$ for all $f\in A+XB[X]$ with $\deg f\leq n$. Because $A$ has ACCP then from Theorem \ref{t3.2} we get $A+XB[X]$ has ACCP. Conversely, because $A+XB[X]$ has ACCP then $A$ has ACCP.	
\end{proof}

In papers \cite{Matysiak1} and \cite{Matysiak2}, polynomial composites with the property of atomicity and ACCP are presented. The results below are complementary.

\begin{pr}
	\label{p3.5}
	Let $T=K+XL[X]$, where $K$, $L$ are fields with $K\subset L$. Let $D$ be a subring of $K$ and $R=D+XL[X]$. Then:
	\begin{itemize}
		\item[(a) ] $R$ is atomic if and only if $T$ is atomic and $D$ is a field.
		\item[(b) ] $R$ satisfies ACCP if and only if $T$ satisfies ACCP and $D$ is a field.
	\end{itemize}
\end{pr}

\begin{proof}
	First suppose that $D$ is not a field. Then $f=d\dfrac{m}{f}$ for each $f\in XL[X]$ and $d\in D^{\ast}$. Thus no element of $XL[X]$ is irreducible ($XL[X]$ is a maximal ideal of $T$). Hence if $R$ is either atomic or satisfies ACCP, $D$ must be a field. So let $D$ be a field.
	
	\medskip
	
	\noindent
	(a) \ \ Up to multiplication by a $\alpha\in K^{\ast}$ (resp. $\alpha\in D^{\ast}$), each element of $T$ (resp. $R$) has the form $f$ or $1+f$ for some $f\in XL[X]$. Each of these elements is irreducible in $R$ if and only if it is irreducible in $T$ (\cite{16}, Lemma 1.5; 27). If $x$ is a product of irreducibles, we may assume that each irreducible factor has the form $f$ or $1+f$ for some $f\in XL[X]$. Thus $x$ is a product of irreducible elements in $R$ if and only if it is a product of irreducible elements in $T$. Hence $R$ is atomic if and only if $T$ is atomic.
	
	\medskip
	
	(b) \ \ We first observe that a principal ideal of $R$ or $T$ may be generated by either $f$ or $a+f$ for some $f\in XL[X]$. Let $f$, $g\in XL[X]$. It easily verified that $(1+f)R\subset (1+g)R$ if and only if $(1+f)T\subset (1+g)T$, $fR\subset (1+g)R$ if and only if $fT\subset (1+g)T$, and $fR\subset gR$ if and only if $fT\subset gT$. Also, if $fT\subset gT$, then $fR\subset (\alpha g)R$ for some $\alpha\in K^{\ast}$. Hence, to each chain of principal ideals of length $s$ in $R$ starting at $fR$ (resp., $(1+f)R$), there corresponds a chain of principal ideals of length $s$ in $T$ starting at $fT$ (resp., $(1+f)T$), and conversely. Thus $R$ satisfies ACCP if only if $T$ satisfies ACCP.
\end{proof}

Propositions \ref{p3.6}, \ref{p3.7} represent the BFD property in polynomial composites.

\begin{pr}
	\label{p3.6} 
	If $A+XB[X]$ is a Noetherian domain, where $A\subset B$ are domains, then $A+XB[X]$ is a BFD.
\end{pr}

\begin{proof}
	\cite{0}, Proposition 2.2.
\end{proof}

\begin{pr}
	\label{p3.7} 
	Let $T=K+XL[X]$, where $K\subset L$ are fields. Let $D$ be a subring of $K$ and $R=D+XL[X]$. Then $R$ is a BFD if and only if $T$ is a BFD and $D$ is a field.
\end{pr}

\begin{proof}
	First suppose that $R$ is BFD. Then $D$ must be a field (\cite{0}, Proposition 1.2). Again from the proof of (\cite{0}, Proposition 1.2) we get that $R$ is a BFD if and only if $T$ is a BFD.  
\end{proof}

Propositions \ref{p3.8} and \ref{p3.9} represents the HFD property in polynomial composites.

\begin{pr}
	\label{p3.8} 
	Let $T=K+XL[X]$, where $K\subset L$ are fields. Let $D$ be a subring of $K$ and $R=D+XL[X]$. Then $R$ is a HFD if and only if $D$ is a field and $T$ is a HFD. 
\end{pr}

\begin{proof}
	As in Proposition 1.2 (\cite{0}), $D$ is necessarily a field. The proof of Proposition 1.2 shows that a factorization into irreducibles in $R$ has the same length as such a factorization in $T$. Hence $R$ is a HFD if and only if $T$ is a HFD. 
\end{proof}

\begin{pr}
	\label{p3.9}
	Let $A$ be a subring of a field $K$. Then $R=A+XK[X]$ is a HFD if and only if $A$ is a field.
\end{pr}

\begin{proof}
	($\Rightarrow$) Clearly, $R$ a HFD implies that $A$ is a HFD. Suppose that $A$ is not a field, so there is an irreducible element $a\in A$. Then $X=a^n(X/a^n)$ for all $n\in\mathbb{N}$. Thus $A$ must be a field.
	
	\medskip
	
	($\Leftarrow$) Suppose that $A$ is a field. By \cite{Matysiak1} Theorem 2.1, $R=A+XK[X]$ is atomic. The proof of Theorem 2.1 \cite{Matysiak1} shows that an irreducible element of $R$ is of the for $aX$, where $a\in K$ or $a(1+Xf[X])$, where $a\in A$, $f(X)\in K[X]$, and $1+Xf(X)$ is irreducible in $K[X]$. Thus for any $g(X)\in R$, the number of irreducible factors from $R$ is the same as the number of irreducible factors in a representation of $g(X)$ as a product of irreducible factors from the PID $K[X]$. Hence $R$ is a HFD.
\end{proof}

Recall that $R$ is an idf-domain if each nonzero element of $R$ has at most a finite number of nonassociate irreducible divisors. 

\begin{pr}
	\label{p3.10}
	Let $T=K+XL[X]$, where $K\subset L$ are fields. Let $M$ be a subfield of $K$ and $R=M+XL[X]$. Then:
	\begin{itemize}
		\item[(a) ] Suppose that $XL[X]$ contains an irreducible element. Then $R$ is an idf-domain if and only if $T$ is an idf-domain and the multiplicative group $K^{\ast}/M^{\ast}$ is finite.
		\item[(b) ] Suppose that $XL[X]$ contains no irreducible elements. Then $R$ is an idf-domain if and only if $T$ is an idf-domain.
	\end{itemize}
\end{pr}

\begin{proof}
	(a)\ \  We first note that an element of $XL[X]$ is irreducible in $R$ if and only if it is irreducible in $T$. Let $f\in XL[X]$ be irreducible. First suppose that $R$ is an idf-domain. Then $af\mid f^2$ for all $a\in K^{\ast}$. Note that $af$ and $bf$ are irreducible in both $R$ and $T$, and that they are associates in $R$ if and only if $a$ and $b$ lie in the same coset in $K^{\ast}/M^{\ast}$. Hence $K^{\ast}/M^{\ast}$ if finite. Let $y\in T$. By multiplying by a suitable $a\in K^{\ast}$, we may assume that $y\in R$. Let $y_1, y_2, \dots, y_n$ be the distinct nonassociate irreducible divisors of $y$ in $R$. It is easily verified that any irreducible divisor of $y$ in $T$ is associated to one of the $y_i$'s. Thus $T$ is also an idf-domain. Conversely, suppose that $T$ is an idf-domain and that $K^{\ast}/M^{\ast}$ is finite. Let $z\in R$. Let $z_1, z_2, \dots, z_r$ be a complete set of nonassociate irreducible divisors of $z$ in $T$, which we may assume are all in $R$, and let $a_1, a_2, \dots, a_s$ be a set of coset representatives of $K^{\ast}/M^{\ast}$. Then any irreducible divisor of $z$ in $R$ is an associate of some $a_iz_j$. Hence $R$ is an idf-domain.
	
	\medskip
	
	\noindent 
	(b)\ \  Since $XL[X]$ has no irreducible elements, an irreducible element in $T$ (resp., in $R$) has the form $a+f$ for some $a\in K^{\ast}$ (resp., $a\in M^{\ast}$) and $f\in XL[X]$. Hence, up to associates, each has the form $1+f$ for some $f\in XL[X]$. It is then easily verified that $\{1+f_1, 1+f_2, \dots, 1+f_n\}$ is a complete set of nonassociate irreducible divisors of a given element with respect to $R$ if and only if it is a complete set of nonassociate irreducible divisors with respect to $T$.
	
\end{proof}

\begin{pr}
	\label{p3.11} 
	Let $T$ be a quasilocal integral domain of the form $K+XL[X]$, where $K\subset L$ are fields. Let $D$ be a subring of $K$ and $R=D+XL[X]$. If $D$ is not a field, then $R$ is an idf-domain if and only if $D$ has only a finite number of nonassociate irreducible elements.
\end{pr}

\begin{proof}
	Let $d$ be a nonzero nonunit of $D$. Then $f=d(f/d)$ shows that no element of $XL[X]$ is irreducible and $d$ divides each element of $XL[X]$. Also, $y=d+f=d(1+f/d)$ and $1+f/d\in R^{\ast}$ (since $T$ is quasilocal) shows that $y$ is irreducible in $R$ if and only if $d$ is irreducible in $D$. Thus $R$ is an idf-domain if and only if $D$ has only a finite number of nonassociate irreducible elements.
\end{proof}

\noindent 
{\bf Question} If $R$ is an idf-domain, then $R[X]$ be an idf-domain?

\medskip

The proposition \ref{p3.12} represents the FFD property in polynomial composites.

\begin{pr}
	\label{p3.12} 
	Let $T=K+XL[X]$, where $K\subset L$ are fields. Let $D$ be a subring of $K$ and $R=D+XL[X]$. Then $R$ is a FFD if and only if $T$ is a FFD, $D$ is a field, and $K^{\ast}/D^{\ast}$ is finite.
\end{pr}

\begin{proof}
	Proof is similar to \cite{0} Proposition 5.2.
\end{proof}

Recall an integral domain $D$ is called an S-domain if for each prime ideal $P$ of $D$ with $ht P=1$, $ht P[X]=1$.

\begin{lm}
	\label{l3.13}
	For an integral domain $D$, the following statements are equivalent.
	\begin{itemize}
		\item[(a) ] $D$ is an S-domain.
		\item[(b) ] For each prime ideal $P$ of $D$ with $ht P=1$, $D_P$ is an S-domain.
		\item[(c) ] For each prime ideal $P$ of $D$ with $ht P=1$, $\overline{D_P}$ is a Pr\"ufer domain. 
	\end{itemize}
\end{lm}

\begin{proof}
	\cite{1} Lemma 3.1
\end{proof}

\begin{lm}
	\label{l3.14}
	For any integral domain $D$, $D[X]$ is an S-domain.
\end{lm}

\begin{proof}
	\cite{1}, Theorem 3.2.
\end{proof}

\begin{pr}
	\label{p3.15} 
	Let $D$ be an integral domain and $S$ a multiplicatively closed subset of $D$. Then $D+XD_S[X]$ is an S-domain.
\end{pr}

\begin{proof}
	Let $R=D+XD_S[X]$ and let $P$ be a height-one prime ideal of $R$. First suppose that $P\cap S\neq\emptyset$. Then $P\supseteq XD_S[X]P=XD_S[X]$. But since $ht P=1$, $P=XD_S[X]$. But then $P\cap S=\emptyset$, a cotradiction. Thus we must have $P\cap S=\emptyset$. Then $P_S$ is a height-one prime ideal in $R_S=D_S[X]$. By Lemma \ref{l3.14}, $R_S$ is an S-domain. Hence $R_P=R_{S_{P_S}}$ is also an S-domain by Lemma \ref{l3.13} (a)$\Rightarrow$(b). Thus $R$ is an S-domain by Lemma \ref{l3.13} (b)$\Rightarrow$(a).
\end{proof}

\medskip

Recall, a commutative ring $R$ is called a Hilbert ring if every prime ideal of $R$ is an intersection of maximal ideals of $R$. In \cite{26} it was shown that if $D\subseteq K$, where $K$ is a field, then $D+XK[X]$ is a Hilbert domain if and only if $D$ is a Hilbert domain. Thus if $D$ is a PID that is not a field and $K$ is the quotient field of $D$, then $D+XK[X]$ is a two-dimensional, non-Noetherian, B\'ezout-Hilbert domain in which every maximal ideal is principal.

\begin{tw}
	\label{t3.16} 
	Let $D$ be an integral domain and $S$ a multiplicatively closed subset od $D$ with the property that for a prime $P$ of $D$ with $P\cap S\neq\emptyset$, then $Q\cap S\neq\emptyset$ for each prime $0\neq Q\subseteq P$. Then $R=D+XD_S[X]$ is a Hilbert domain if and only if $D$ and $D_S$ are Hilbert domains.
\end{tw}

\begin{proof}
	($\Rightarrow$) Suppose that $R$ is a Hilbert domain. Then $D\cong R/XD_S[X]$ is also a Hilbert domain. Suppose that $D_S$ is not a Hilbert domain. Let $Q$ be a nonzero prime ideal od $D$ with $Q\cap S\emptyset$. Since $D$ is a Hilbet domain, $Q=\bigcap_{\alpha} M_{\alpha S}$, where $\{M_{\alpha}\}$ is the set of maximal ideals of $D$ containing $Q$. Since $Q\cap S=emptyset$ by the hypothesis on $S$, each $M_{\alpha}\cap S=\emptyset$. Hence $Q_S=\bigcap M_{\alpha S}$ is an intersection of maximal ideals of $D_S$. So every nonzero prime ideal of $D_S$ is an intersection of maximal ideals. Hence there is a nonzero element $u\in D$ such that $u$ is in every nonzero prime ideal of $D_S$. Consider $u+X\in R$. Let $P$ be prime ideal of $R$ minimal over ($u+X$) with $P\cap D=0$. (Such a prime $P$ exists since $(u+X)\cap(D\setminus\{0\})=\emptyset$). If $Q$ is a prime ideal of $R$ with $P\subsetneq Q$, then $Q\cap D\neq 0$. For otherwise in $D_S[X]$, $0\neq P_S\subsetneq Q_S$ would both contract to $0$. Now if $Q\cap S\neq\emptyset$, then $X\in XD_S[X]\subseteq Q$, while if $Q\cap S=\emptyset$, then $u\in(Q_S\cap D_S)\cap D\subseteq Q$. So every prime ideal of $R$ properly containing $P$ contains both $u$ and $X$. Hence $P$ is not the intersection of the maximal ideals containing it, contradicting the fact that $R$ is a Hilbert domain. So $D_S$ must also be a Hilbert domain.
	
	\medskip
	
	($\Leftrightarrow$) Let $Q$ be a prime ideal of $R$. Suppose that $Q\cap S\neq\emptyset$. Then $XD_S[X]=XD_S[X]Q\subseteq Q$, so $Q=Q\cap D+XD_S[X]$. Since $D$ is a Hilbert domain, $Q\cap D$ is an intersection of maximal ideals, hence so is $Q$. So we may suppose that $Q\cap S=\emptyset$. Then since $D_S[X]$ is a Hilbert domain, $Q_S=\bigcap_{\alpha} M_{\alpha}$, where $\{M_{\alpha}\}$ is the set of maximal ideals of $D_S[X]$ containing $Q_S$. Then $Q=\bigcap_{\alpha}(M_{\alpha}\cap R)$. So it suffices to show that each $M_{\alpha}\cap R$ is a maximal ideal of $R$. So let $M$ be a maximal ideal of $D_S[X]$. Then $M=N_S$, where $N$ is a prime ideal of $D[X]$. Now $M$ maximal implies $M\cap D_S$ is maximal since $D_S$ is Hilbert domain. If $M\cap D_S=0$, then $D_S$ is a field and hence $R$ is a Hilbert domain (\cite{26}, Theorem 5). So we may assume that $M\cap D_S\neq 0$. Then by hypothesis on $S$, $(M\cap D_S)\cap D=N\cap D$ must also be maximal. Since $N\supsetneq (N\cap D)[X]$, $N$ must be a maximal ideal of $D[X]$. Hence $D[X]/N\subseteq R/M\cap R\subseteq D_S[X]/M=D_S[X].N_S=D[X]/N$ since $D[X]/N$ is a field. Therefore $M\cap R$ is a maximal ideal.  
\end{proof}

The next Proposition says that every polynomial composite is a one-dimensional B\'ezout domain.

\begin{pr}
	\label{p3.17}
	Let $K\subset L$ be a pair of fields with $L$ purely inseparable over $K$ (that is, $char K=p>0$ and for each $l\in L$, there exists a natural number $n=n(l)$ with $l^{p^n}\in K$). Then every ring $R$ between $K[X]$ and $L[X]$ is a one-dimensional almost B\'ezout domain.
\end{pr}

\begin{proof}
	Since $K[X]\subset L[X]$ is an integral extension, $\dim R=\dim K[X]=1$. For each $f\in L[X]$, $f^{p^n}\in K[X]$ for $n$ large enough. Hence for $f, g\in R$, $f^{p^n}, g^{p^n}\in K[X]$ for some $n\in\mathbb{N}$. But $(f^{p^n}, g^{p^n})K[X]$ is principal. Hence $(f^{p^n}, g^{p^n})R$ is principal.
\end{proof}

Let $K$ be a field, $D$ a subring of $K$.
Every ring $R$ between $D[X]$ and $K[X]$ has a composite cover (\cite{1} Section II), i.e. the unique minimal overring of $R$ that is a composite.
Recall $I(B,A)=\{f(X)\in B[X]\mid f(A)\subseteq A\}$.

\begin{pr}
	\label{p3.18} 
	(a) Let $R$ be a domain with qoutient field $K$. Suppose that for each $0\neq r\in R$, $R/(r)$ is finite. Then the composite cover of $I(K,R)$ is $R+XK[X]$. 
	
	\medskip
	
	\noindent 
	(b) Let $A\subseteq B$ be rings where $A$ is finite. Then the composite cover of $I(B,A)$ is $A+XB[X]$. 
\end{pr}

\begin{proof}
	(a) Let $r$ be a nonzero nonunit of $R$ and let $R/(r)=\{r_1+(r), \dots, r_n+(r)\}$. Set $f(X)=\dfrac{1}{r}(X-r_1)\dots (X-r_n)\in K[X]$. Now for $a\in R$, $a+(r)=r_i+(r)$ for some $i$, so $a-r_i=sr$ for some $s\in R$. Hence $f(a)=\dfrac{1}{r}(sr)\prod_{j\neq 1}(a-r_j)\in R$. So $f(X)=\dfrac{1}{r}X^n+\dots \in I(K,R)$ and hence $I(K,R)$ has composite cover $R+XK[X]$.
	
	\medskip
	
	\noindent 
	(b) For each $b\in B$, $f(X)=b(\prod_{a\in A}(X-a))\in I(B,A)$.
\end{proof}

\section{Characterization of field extensions in terms of polynomial composites}
\label{R4}

Let's start with an auxiliary lemmas that will help.

\begin{lm}
	\label{l4.1}
	Let $\varphi\colon K_1\to K_2$ be an isomorphism of fields and $\varPsi\colon K_1[X]\to K_2[X]$ be an isomorphism of polynomials ring. If polynomial $f_1\in \Irr K_1[X]$, $f_1$ has a root $a_1$ in an extended $L_1$ of $K_1$ and polynomial $f_2=\varPsi (f_1)$ has a root $a_2$ in an extended $L_2$ of $K_2$, then there exists $\varPsi'\colon K_1(a_1)\to K_2(a_2)$, which is an extension of $\varphi$ and $\varPsi'(a_1)=a_2$ holds.
\end{lm}

\begin{proof}
	\cite{TC}, Lemma 2, p. 105.
\end{proof}

\begin{lm}
	\label{l4.2}
	If $\varphi\colon K\to L$ be an embedding field $K$ to an algebraically closed field $L$, and $K'$ be an algebraic extension of $K$, then there exists an embedding $\varPsi\colon K'\to L$ which is an extension of $\varphi$.
\end{lm}

\begin{proof}
	\cite{TC}, Lemma 4, p. 109.
\end{proof}

\begin{lm}
	\label{l4.3}
	If $L$ be a finite field extension of $K$, then $L$ is a Galois extension of $K$ if and only if $|G(L|K)|=(L\colon K)$.
\end{lm}

\begin{proof}
	\cite{TC}, Corollary 1, p. 126.
\end{proof}

\begin{lm}
	\label{l4.4}
	If there exists a nonzero ideal $I$ of $L[X]$, where $L$ is a field, that is finitely generated as an $K+XL[X]$-module, then $K$ is a field and $[L\colon K]<\infty$.
\end{lm}

\begin{proof}
	Clearly, $I$ is finitely generated over $L[X]$, and hence $XL[X]I\neq I$. For otherwise, $XL[X]L[X]_{XL[X]}\cdot IL[X]_{XL[X]}=IL[X]_{XL[X]}$ and therefore $IL[X]_{XL[X]}=0$, by Nakayama's lemma. This is impossible, since $0\neq I\subseteq IL[X]_{XL[X]}$. It follows that $I/XL[X]I$ is a nonzero ($L[X]/XL[X]=L$)-module that is finitely generated as an ($K+XL[X]/XL[X]=K$)-module. Since $L$ is a field, $I/XL[X]I$ can be written as a direct sum of copies of $L$. Thus, $L$ is a finitely generated $K$-module. But then $K$ is a field, since the field $L$ is integral over $K$ and obviously $[L\colon K]<\infty$.
\end{proof}

\begin{tw}
	\label{t4.5}
	Let $K\subset L$ be a field extension. Put $T=K+XL[X]$. Then the following conditions are equivalent:
	\begin{itemize}
		\item[(1) ] $T$ is Noetherian.
		\item[(2) ] $[L\colon K]<\infty$.
	\end{itemize}
\end{tw}

\begin{proof}
	$(1)\Rightarrow (2)$
	Since $XL[X]$ is a finitely generated ideal of $K+XL[X]$, it follows from Lemma \ref{l4.4} that $[L\colon K]<\infty$. Thus, $L[X]$ is module-finite over the Noetherian ring $K+XL[X]$.
	
	\medskip
	
	\noindent 
	$(2)\Rightarrow (1)$
	$L[X]$ is Noetherian ring and module-finite over the subring $K+XL[X]$. This is the situation covered by P.M. Eakin's Theorem \cite{zzz}.
\end{proof}

All our considerations began with the Theorem \ref{t4.5}. This Theorem motivated us to further consider polynomial composites $K+XL[X]$ in a situation where the extension of fields $K\subset L$ is algebraic, separable, normal and Galois, respectively.

\begin{pr}
	\label{p4.6}
	Let $K\subset L$ be a fields extension such that $L^{G(L\mid K)}=K$. Put $T=K+XL[X]$. Then the following conditions are equivalent:
	\begin{itemize}
		\item[(1) ] $T$ is Noetherian.
		\item[(2) ] $K\subset L$ be an algebraic extension.
	\end{itemize}
\end{pr}

\begin{proof}
	$(1)\Rightarrow (2)$\ \
	Since $T=K+XL[X]$ is Noetherian, where $K\subset L$ be fields extension, then by Theorem \ref{t4.5} we get that $K\subset L$ is a finite extension. And every finite extension is algebraic. 
	
	\medskip
	
	\noindent 
	$(2)\Rightarrow (1)$\ \ 
	Assume that $K\subset L$ be an algebraic extension. Assuming $L^{G(L\mid K)}=K$ we get directly from the definition of the Galois extension. Since $K\subset L$ be the Galois extension, then $K\subset L$ be a normal extension. Every normal extension is finite, then by Theorem \ref{t4.5} we get that $K+XL[X]$ be a Noetherian.
\end{proof}

\begin{pr}
	\label{p4.7}
	Let $K\subset L$ be fields extension such that $K$ be a perfect field and assume that any $K$-isomorphism $\varphi\colon M\to M$, where $\varphi(L)=L$ holds for every field $M$ such that $L\subset M$.
	Put $T=K+XL[X]$.
	\begin{itemize}
		\item[(1) ] $T$ is Noetherian.
		\item[(2) ] $K\subset L$ be a separable extension.
	\end{itemize} 
\end{pr}

\begin{proof}
	$(1)\Rightarrow (2)$ \ \
	By Theorem \ref{t4.5} $K\subset L$ be a finite extension. Every finite extension be an algebraic extension. Since $K$ be the perfect field, then $K\subset L$ be a separable extension.
	
	\medskip
	
	\noindent 
	$(2)\Rightarrow (1)$\ \ 
	First we show that if $L$ be a separable extension of the field $K$, then the smallest normal extension $M$ of the field $K$ containing $L$ be the Galois extension of the field $K$.
	
	\medskip
	
	If $L$ be a separable extension of the field $K$, and $N$ be a normal extension of the field $K$ containing $L$, then let $M$ be the largest separable extension of $K$ contained in $N$. So we have $L\subset M$ and therefore it suffices to prove that $M$ be the normal extension of $K$. 
	
	\medskip
	
	Let $g\in\Irr K[X]$ has a root $a$ in the field $M$. Because $N$ be the normal extension of $K$ and $a\in N$, so it follows that all roots of polynomial $G$ belong to the field $N$. The element $a$ is separable relative to $K$, and so belong to $M$. Hence polynomial $g$ is the product of linear polynomials belonging to $M[X]$, which proves that $M$ be the normal extension of the field $K$.
	
	\medskip
	
	Since $M$ be the normal extension of $K$ and the Galois extension of $K$, then $L$ be the normal extension of $K$ by the assumption (\cite{TC}, Exercise 4, p. 119).
	
	\medskip
	
	Because $L$ be the normal extension of $K$, then $L$ be the finite extension of $K$. And by Theorem \ref{t4.5} we get that $K+XL[X]$ is Noetherian.
\end{proof}

\begin{pr}
	\label{p4.8}
	Let $K\subset L$ be fields extension and let $T=K+XL[X]$. Assume that if a map $\varphi\colon L\to a(K)$ is $K$-embedding, then $\varphi (L)=L$. Then the following conditions are equivalent:
	\begin{itemize}
		\item[(1) ] $T$ is Noetherian.
		\item[(2) ] $K\subset L$ is a normal extension.
	\end{itemize}
\end{pr}

\begin{proof}
	$(1)\Rightarrow (2)$ \ \
	By Proposition \ref{p4.6} $K\subset L$ be the algebraic extension.
	
	\medskip
	
	Let $c$ be a root of polynomial $g$ belonging to $L$, and $b$ be the arbitrary root of $g$ belonging to $a(K)$. Because polynomial $g\in\Irr K[X]$, so by Corollary from Lemma \ref{l4.1} there exists $K$-isomorphism $\varphi'\colon K(c)\to K(d)$. By Lemma \ref{l4.2} it can be extended to embedding $\varphi\colon L\to a(K)$. Hence towards $\varphi (L)=L$ and $\varphi (K(c))=\varphi'(K(c))=K(d)$ we get that $K(d)\subset L$, so $b\in L$. Hence every root of polynomial $g$ belong to $L$, so polynomial $g$ is the product of linear polynomials belonging to $L[X]$ ($\ast$). 
	
	\medskip
	
	For every $c\in L$ let $g_c\in\Irr K[X]$ satisfying $g_c(c)=0$. By ($\ast$) every roots of $g_c$ belong to the field $L$. Hence $L$ is a composition of splitting field of all polynomials $g_c$, where $c\in L$. Hence $L$ be the normal extension of $K$.
	
	\medskip
	
	\noindent 
	$(2)\Rightarrow (1)$\ \
	If $L$ be a normal field extension of the field $K$, then $K\subset L$ be the finite extension. Then by Theorem \ref{t4.5} we get that $K+XL[X]$ be Noetherian.
\end{proof}

In the above Proposition we can replace the assumption that "if a map $\varphi\colon L\to a(K)$ is $K$-embedding, then $\varphi (L)=L$" to "$L^{G(L\mid K)}=K$". Then:

\begin{pr}
	\label{p4.9}
	Let $K\subset L$ be fields extension such that $L^{G(L\mid K)}=K$. If $T=K+XL[X]$ is a Noetherian, then $K\subset L$ be a normal extension.
\end{pr}

\begin{proof}
	By Proposition \ref{p4.6} we get that $K\subset L$ is the algebraic field extension. Assuming $L^{G(L\mid K)}=K$ we get directly from the definition of the Galois extension, and so normal extension. 
\end{proof}

\begin{pr}
	\label{p4.10}
	Let $T=K+XL[X]$ be Noetherian, where $K\subset L$ be fields. Assume
	$|G(L\mid K)|=[L\colon K]$ and any $K$-isomorphism $\varphi\colon M\to M$, where $\varphi(L)=L$ holds for every field $M$ such that $L\subset M$.
	Then the following conditions are equivalent:
	\begin{itemize}
		\item[(1) ] $T$ is a  Noetherian.
		\item[(2) ] $K\subset L$ be a Galois extension. 
	\end{itemize} 
\end{pr}	

\begin{proof}
	$(1)\Rightarrow (2)$\ \
	By Theorem $\ref{t4.5}$ we get $K\subset L$ be the finite extension. By the assumption we can use Lemma \ref{l4.3} and we get that $K\subset L$ be Galois extension.
	
	\medskip 
	
	\noindent 
	$(2)\Rightarrow (1)$ \ \
	If $K\subset L$ be Galois fields extension, then is separable. By Proposition \ref{p4.7} we get that $K+XL[X]$ be Noetherian.
\end{proof}

In the above Proposition we can swap the assumptions that "$|G(L\mid K)|=[L\colon K]$ and any $K$-isomorphism $\varphi\colon M\to M$, where $\varphi(L)=L$ holds for every field $M$ such that $L\subset M$" to "$L^{G(L\mid K)}=K$". Then:

\begin{pr}
	\label{p4.11}
	Let $T=K+XL[X]$, where $K\subset L$ be fields such that $K=L^{G(L\mid K)}$.
	The following conditions are equivalent:
	\begin{itemize}
		\item[(1) ] $T$ is a Noetherian.
		\item[(2) ] $K\subset L$ be a Galois extension. 
	\end{itemize}  
\end{pr}	

\begin{proof}
	$(1)\Rightarrow (2)$\ \
	By Proposition $\ref{p4.6}$ we get that $K\subset L$ be the algebraic extension. Assuming $K=L^{G(L\mid K)}$ we get directly from the definition of the Galois extension.
	
	\medskip
	
	\noindent 
	$(2)\Rightarrow (1)$\ \
	If $K\subset L$ be Galois fields extension. Then $K\subset L$ be a normal extension. Hence by Proposition \ref{p4.9} we get that $K+XL[X]$ be Noetherian.
\end{proof}

We present the full possible characterization of field extensions. Combining the Magid results and from this paper, we get the following two theorems.

\begin{tw}[\cite{Magid}, Theorem 1.2.]
	\label{t4.12}
	Let $M$ be an algebraically closed field algebraic over $K$, and let $L$ such that $K\subseteq L\subseteq M$ be an intermediate field. Then the following are equivalent:
	\begin{itemize}
		\item[(a) ] $L$ is separable over $K$.
		\item[(b) ] $M\otimes_K L$ has no nonzero nilpotent elements.
		\item[(c) ] Every element of $M\otimes_K L$ is a unit times an idempotent.
		\item[(d) ] As an $M$-algebra $M\otimes_KL$ is generated by idempotents.
	\end{itemize}
\end{tw}

\begin{tw}[\cite{Magid}, Theorem 1.3.]
	\label{t4.13}
	Let $M$ be an algebraically closed field containing $K$, and let $L$ be a field algebraic over $K$. Then the following are equivalent:
	\begin{itemize}
		\item[(a) ] $L$ is separable over $K$.
		\item[(b) ] $M\otimes_K L$ has no nonzero nilpotent elements.
		\item[(c) ] Every element of $M\otimes_K L$ is a unit times an idempotent.
		\item[(d) ] As an $M$-algebra $M\otimes_KL$ is generated by idempotents.
	\end{itemize}
\end{tw}	

Below we have conclusions from the above results.

\begin{tw}
	\label{t4.14}
	In Theorems \ref{t4.12} and \ref{t4.13} if assume $L^{G(L\mid K)}=K$, then the conditions (a) -- (d) are equivalent to
	\begin{itemize}
		\item[(e) ] $K+XL[X]$ be a Noetherian.
		\item[(f) ] $[L\colon K]<\infty$
		\item[(g) ] $K\subset L$ be an algebraic extension.
		\item[(h) ] $K\subset L$ be a Galois extension.
	\end{itemize}
\end{tw}

\begin{proof}
	(h)$\Rightarrow$(a) -- Obvious.
	
	\medskip
	
	(a)$\Rightarrow$(g)$\Rightarrow$(e)$\Rightarrow$(h)\ \ If $K\subset L$ be a separable extension, then be an algebraic extension. By Proposition \ref{p4.6} $K+XL[X]$ be a Noetherian. By Proposition \ref{p4.11} $K\subset L$ be a Galois extension.
	
	\medskip
	
	\noindent 
	(e)$\Rightarrow$(f) -- Theorem \ref{t4.5}.
\end{proof}

\begin{tw}
	\label{t4.15}
	In Theorem \ref{t4.14} if assume that $K$ is a perfect field and $L^{G(L\mid K)}=K$, then the conditions (a) -- (h) are equivalent to 
	
	\medskip 
	
	\noindent  
	(i) $K\subset L$ be a normal extension.
\end{tw}

\begin{proof}
	(i)$\Rightarrow$(a)\ \ If $K\subset L$ be a normal extension, then be an algebraic extension. By definition perfect field $K\subset L$ be a separable extension.
	
	\medskip
	
	(h)$\Rightarrow$(i)\ \ Obvious.
\end{proof}

\section{An SR property in polynomial composites}
\label{R5}

Recall from \cite{Matysiak6} a monoid $H$ is called $SR$-monoid, if $\Sqf H=\Gpr H$.
We present facts based on monoids. They can be easily transferred to integral domains.

\begin{pr}
	\label{p5.1}
	If $H$ is an $AP$-monoid, then $R$ is an $SR$-monoid.
\end{pr}

\begin{proof}
	Let $s\in\Sqf R$. Then $s=q_1\dots q_n$, where $q_1$, $\dots$, $q_n\in\Irr R$, $s_i\rpr s_j$. Since $R$ is $AP$-monoid, then all $q_i\in\Prime R$ and since $s=q_1\dots q_n$ and $q_i\rpr q_j$ then $s\in\Gpr H$.
\end{proof}

The opposite fact generally does not hold.

\begin{ex}
	\label{e5.2}
	Let $H$ be a non-group monoid such that every element of $H$ is a square. For example $\mathbb{Q}_{\geqslant 0}$. Then
	$\Sqf H=\Gpr H=H^{\ast}$, so $H$ is an SR-monoid. But since $\Irr H\subset \Sqf H$, then $\Irr H$ must be empty. Similarly $\Prime H$ is also empty. There is not an $AP$-property. So $H$ is $SR$-monoid, non-$AP$-monoid.
\end{ex}

\begin{pr}
	\label{p5.3}
	Let $R$ be a factorial monoid. Then the following conditions are equivalent:
	\begin{itemize} 
		\item[(a) ] $H$ is an $AP$-monoid,
		\item[(b) ] $H$ is an $SR$-monoid.
	\end{itemize} 
\end{pr}

\begin{proof}
	From Proposition \ref{p5.1} we have (a) $\Rightarrow$ (b). So it sufficient to show that every square-free is radical. Let $s\in\Sqf H$ and $s\mid x^m$ for all $x\in H$ and $m\in\mathbb{N}_{>1}$. Since $H$ is factorial, then $s=q_1\dots q_n$, with $q_i\in\Irr H=\Prime H$ and $q_i\rpr q_j$ and $x=p_1\dots p_r$, with $p_i\in\Irr H=\Prime H$. Then $q_1\dots q_n\mid (p_1\dots p_r)^m$. 
	
	\medskip
	
	We see that for all $i\in\{1, \dots, n\}$ we have $q_i\mid (p_1\dots p_r)^m$. Since $q_i\in\Prime H\subset\Gpr H$, then $q_i\mid p_1\dots p_r$, i.e. $q_i\mid x$. Thus $\lcm(q_1, \dots, q_n)\mid x$ and $q_i\rpr q_j$, then $s\mid x$.
\end{proof}

\begin{cor}
	\label{c5.4}
	$H$ is factorial if and only if $H$ is atomic and $H$ is $SR$-monoid.
\end{cor}

The $SR$-property holds in pre-Schreier monoids, $GCD$-monoids. So it is enough to consider square-free dependencies.

\begin{pr}
	\label{p5.5}
	Let $K$ be a field and $L$ its a field of fractions. Then $K+XL[X]$ is a pre-Schreier domain.
\end{pr}

\begin{proof}
	It follows from definition of pre-Schreier domain and it sufficient to check free coefficients.
\end{proof}

\begin{pr}
	\label{p5.6}
	Let $H$ be a monoid and $S\subset H$ be an multiplicative subset. Then, if $H$ is an $SR$-monoid, then $H_S$ is $SR$-monoid.
\end{pr}

\begin{proof}
	Assume $a\in\Sqf H=\Gpr H$.
	Suppose $a\in\Sqf H_S\setminus\Gpr H_S$. By \cite{Matysiak6}, Proposition 3.1 if $a\in\Sqf H$, then $a\in\Sqf H_S$. By \cite{Reinhart} Lemma 2.1,  if $a\notin\Gpr H_S$, then $a\notin\Gpr H$. A contradiction.
\end{proof}

\begin{cor}
	\label{c5.7}
	If $T=K+XL[X]$ is atomic and $\Irr T\subset\Gpr T$, then  $T$ is radical factorial.
\end{cor}


\begin{thebibliography}{0}

\bibitem{0}
Anderson, D.D., Anderson, D.F., Zafrullah, M., {\em Factorization in integral domains}. J. Pure Appl. Algebra 69, 1–19 (1990).

\bibitem{1}
Anderson, D.D., Anderson, D.F., Zafrullah, M., {\it Rings between D[X] and K[X]}, Houston J. of Mathematics, 17, (1991) 109--129.

\bibitem{AndGotti}
Anderson, D.F., Gotti, F., {\it Bounded and finite factorization domains}, arXiv:2010,02722 (2020).

\bibitem{y5}
Anderson, D.F., Ryckaert, A., {\it The class group of $D + M$}, J. Pure Appl. Algebra, 52, (1988) 199--212.

\bibitem{eco}
Beach, E.F., {\it The use of polynomials to represent cost functions}, The Review of Economic Studies, 16 3, 158 -- 169, (1949).

\bibitem{y1}
Brewer, J., Rutter, E., {\it D + M construction with general overrings}, Mich. Math. J., 23, (1976) 33--42.

\bibitem{TC}
Browkin, J., {\it Field theory }, Biblioteka Matematyczna, 49, (1977).

\bibitem{Cohn}
Cohn, P.M., {\it Bezout rings and their subrings}, Proc. Camb. Phil.Soc., 64, (1968) 251–264.

\bibitem{y2}
Costa, D., Mott, J., Zafrullah, M., {\it The construction $D + XD_S[X]$}, J. Algebra, 153, (1978) 423--439.

\bibitem{16}
Costa, D., Mott, J., Zafrullah, M., {\it Overrings and dimensions of general D + M constructions}, J. Natur. Sci. Math, 26 (1986), 7 -- 14.

\bibitem{amal}
D'Anna, M., Finocchiaro, C., Fontana, M., {\it Amalgamated algebras along an ideal}, Commutative Algebra and Applications, Proceedings of the Fifth International Fez Conference on Commutative Algebra and Applications, Fez, Morocco, 2008, W. de Gruyter Publisher, Berlin (2009), 155-172.

\bibitem{zzz}
P.M. Eakin Jr., {\it The converse to a well known theorem on Noetherian rings}, Math. Ann., 177 (1968), 278 -- 282.

\bibitem{x9}
Eftekhari, S., Khorsandi, M.R., {\em MCD-finite domains and ascent of IDF property in polynomial extensions}, Commun. Algebra 46, 3865–3872 (2018)

\bibitem{y4}
Fontana, M., Kabbaj, S., {\it On the Krull and valuative dimension of $D + XD_S[X]$ domains}, J. Pure Appl. Algebra, 63, (1990) 231--245.

\bibitem{Gil}
Gilmer, R., {\it Multiplicative Ideal Theory}, M. Dekker, New York (1972).

\bibitem{Grams}
A. Grams, {\it Atomic rings and the ascending chain condition for principal ideals}, Proc. Cambridge Philos, Soc., 75 (1974), 321--329.

\bibitem{eco2}
Judge, G., {\it Developing an Interest in Polynomials — An Example from the Economics of Investment Decision Making}, Teaching Mathematics and its Applications: An International Journal of the IMA, 4 3, 127–-129, (1985).

\bibitem{Magid}
Magid, A.R., {\it The Separable Galois Theory of Commutative Rings}, Chapman \& Hall/CRC, Boca Raton, 2014.

\bibitem{Matysiak1}
Matysiak, \L .,
{\em On properties of composites and monoid domains},
Italian Journal of Pure and Applied Mathematics, 47 (2022), 796-811.

\bibitem{Matysiak2}
Matysiak, \L ., {\em ACCP and atomic properties of composites and monoid domains}, (2022), https://lukmat.ukw.edu.pl/files/ACCP-and-atomic-properties-of-composites-and-monoid-domains.pdf

\bibitem{Matysiak4}
Matysiak, \L . {\em Polynomial composites and certain types of fields extensions}, (2021), accepted in 2022, https://lukmat.ukw.edu.pl/files/Polynomial-composites-and-certain-types-of-fields-extensions.pdf	

\bibitem{Matysiak3}
Matysiak, \L . {\em On some properties of polynomial composites}, (2021), accepted in 2022, https://lukmat.ukw.edu.pl/files/On-some-properties-of-polynomial-composites.pdf

\bibitem{Matysiak5}
Matysiak, \L ., {\em On square-free and radical factorizations and existence of some divisors}, (2021) https://lukmat.ukw.edu.pl/files/On-square-free-and-radical-factorizations-and-existence-of-some-divisors.pdf

\bibitem{Matysiak6}
Matysiak, \L ., 
{\em On square-free and radical factorizations and existence of some divisors and relationships with the Jacobian conjecture}, (2021), accepted in 2022, https://lukmat.ukw.edu.pl/files/On-square-free-and-radical-factorizations.pdf

\bibitem{Matysiak7}
Matysiak, \L ., 
{\em An SR-property in integral domains and monoids}, (2022)

\bibitem{26}
Mott, J., Zafrullah, M., {\it Unruly Hilbert domains}, Canad. Math. Bull, 33 (1990), 106 -- 109. 

\bibitem{Reinhart}
Reinhart, A.,  
{\em Radical factorial monoids and domains}, Ann. Sci. Math. Qu\'ebec, v. 36 (2012), 193--229.

\bibitem{y3}
Zafrullah, M., {\it The $D + XD_S[X]$ construction from GCD-domains}, J. Pure Appl. Algebra, 50, (1988) 93--107.	
	
\end{thebibliography}
\end{document}